\documentclass[a4paper,11pt]{amsart}
\pdfoutput=1
\usepackage[utf8]{inputenc}
\usepackage[T1]{fontenc}  
\usepackage{lmodern}
\usepackage{url}
\usepackage{amsmath,amssymb,amsfonts,amsthm}
\usepackage{pstricks,pst-tree,pst-node} 
\usepackage{pst-plot}
\usepackage{array}
\usepackage{enumerate} 
\usepackage{fancyhdr}
\usepackage{pst-circ}
\usepackage{mathrsfs}
\usepackage{float}
\usepackage{graphics}
\usepackage{multirow}
\usepackage{caption}
\usepackage{xcolor}

\renewcommand{\tilde}[1]{\widetilde{#1}}

\newcommand*{\vect}[1]{\mathrm{Vect}\{#1\}}

\newcommand*{\vcd}{\mathrm{vcd}}
\newcommand*{\cd}{\mathrm{cd}}
\newcommand*{\gd}{\mathrm{gd}}
\newcommand*{\SL}{\mathrm{SL}}

\newcommand*{\SO}{\mathrm{SO}}

\newcommand*{\PSL}{\mathrm{PSL}}
\newcommand*{\PGL}{\mathrm{PGL}}

\newcommand*{\Sp}{\mathrm{Sp}}
\newcommand*{\U}{\mathrm{U}}

\newcommand*{\syst}{\mathrm{syst}}
\newcommand*{\Aut}{\mathrm{Aut}}

\newtheorem{thm}{Theorem}[section]

\newtheoremstyle{mth}{0.15cm}{0.15cm}{\itshape}{}{\scshape}{}{0.5em}{\textbf{\thmname{#1}}}
\theoremstyle{mth}

\newtheoremstyle{de}{0.15cm}{0.15cm}{\upshape}{}{\upshape}{}{0.5em}{\underline{\textbf{\thmname{#1} \thmnumber{#2}}} #3}
\theoremstyle{de}

\newtheoremstyle{rmq}{0.15cm}{0.15cm}{}{}{\itshape}{}{0.5em}{\thmname{#1} \thmnumber{#2} #3}
\theoremstyle{rmq}
\newtheorem{rmq}{Remark}[section]
\newtheoremstyle{ex}{0.15cm}{0.15cm}{}{}{\itshape}{}{0.5em}{\thmname{#1} \thmnumber{#2} #3}
\theoremstyle{ex}

\author{Cyril Lacoste}

\title{On the difficulty of finding spines}

\makeatletter

\@addtoreset{equation}{section}
\makeatother
\begin{document}

\maketitle

\noindent \textbf{Abstract.} We prove that the set of symplectic lattices in the Siegel space $\mathfrak{h}_g$ whose systoles generate a subspace of dimension at least 3 in $\mathbb{R}^{2g}$ does not contain any $\Sp(2g,\mathbb{Z})$-equivariant deformation retract of $\mathfrak{h}_g$.

\vspace{0.5 cm}

\noindent \textbf{Résumé.} Nous montrons que l'ensemble des réseaux symplectiques dans l'espace de Siegel $\mathfrak{h}_g$ dont les systoles engendrent un sous-espace de dimension au moins 3 dans $\mathbb{R}^{2g}$ ne contient aucun rétract par déformation $\Sp(2g,\mathbb{Z})$-équivariant de $\mathfrak{h}_g$.

\vspace{0.5 cm} 

\noindent \textbf{Version française abrégée.} Soit $\Gamma$ un groupe discret infini. On dit qu'un $\Gamma$-complexe cellulaire $W$ est un modèle de type $\underline{E}\Gamma$ si pour tout sous-groupe $H\subset \Gamma$, l'ensemble de points fixes $W^H$ est contractile si $H$ est fini et vide sinon. La plus petite dimension possible d'un tel espace est la dimension géométrique propre de $\Gamma$, notée $\underline  {\gd}(\Gamma)$. Nous avons montré dans \cite{Lacoste}, en nous basant sur les résultats de \cite{Aramayona}, que si $G$ est un groupe de Lie linéaire semisimple et $\Gamma \subset G$ un réseau, alors $\underline{\gd}(\Gamma)$ est égale à la dimension cohomologique virtuelle $\vcd(\Gamma)$ de $\Gamma$, c'est-à-dire la dimension cohomologique de n'importe quel sous-groupe d'indice fini sans torsion. Si $K \subset G$ est un sous-groupe compact maximal, l'espace symétrique $X=G/K$ est un modèle de type $\underline{E}\Gamma$, mais pas de dimension minimale sauf si $\Gamma$ est cocompact. Nous souhaitons ainsi trouver concrètement un modèle de type $\underline{E}\Gamma$ de dimension $\vcd(\Gamma)$. Un tel espace peut par exemple servir à calculer la cohomologie de $\Gamma$.

Il est naturel d'essayer de construire ce modèle en tant que sous-ensemble de l'espace symétrique $X$, dans ce cas nous l'appelons une épine. Plus précisément, une épine pour $\Gamma$ est un rétract par déformation $\Gamma$-équivariant de l'espace symétrique $X=G/K$, de dimension $\vcd(\Gamma)$, sur lequel $\Gamma$ agit de manière cocompacte.

Il peut paraître surprenant que de tels espaces ne sont connus que dans très peu de cas (groupes de $\mathbb{Q}$-rang 1, $\SL(n,\mathbb{Z})$). Le but de cet article est d'expliquer pourquoi il peut être difficile de trouver des épines.

Rappelons brièvement la construction de l'épine du groupe $\SL(n,\mathbb{Z})$. On identifie l'espace symétrique $S_n=\SL(n,\mathbb{R})/\SO(n)$ avec l'ensemble des réseaux de $\mathbb{R}^n$ de covolume 1 modulo isométries munis d'une $\mathbb{Z}$-base. Généralisant un résultat de Soulé (voir \cite{Soulethese}), Ash a prouvé dans \cite{Ash} que l'ensemble des réseaux dont les systoles (ou vecteurs minimaux) engendrent $\mathbb{R}^n$ est une épine pour $\SL(n,\mathbb{Z})$ (voir \cite{Ash}). Plus précisément, si $\mathcal{X}_i$ est l'ensemble des réseaux dont les vecteurs minimaux engendrent un sous-espace de dimension au moins $i$ (pour $i=1,\dots,n$), on peut montrer que $\mathcal{X}_{i+1}$ est un rétract par déformation $\SL(n,\mathbb{Z})$-équivariant de $\mathcal{X}_i$ pour tout $i=1,\dots,n-1$.

Il y a eu par la suite des tentatives de constructions similaires pour d'autres groupes arithmétiques tels que le groupe symplectique $\Sp(2g,\mathbb{Z})$. Par exemple,  Bavard a montré dans \cite{Bavard2} que l'ensemble des réseaux symplectiques dont les systoles engendrent un sous-espace non-isotrope pour la forme symplectique est un rétract $\Sp(2g,\mathbb{Z})$-équivariant de l'espace de Siegel $\mathfrak{h}_g=\Sp(2g,\mathbb{R})/\U(g)$. Malheureusement ce rétract est de codimension 1, car il existe des réseaux symplectiques avec seulement deux systoles non-isotropes. On pourrait alors essayer de  rétracter sur l'ensemble des réseaux symplectiques avec au moins trois systoles linéairement indépendantes. Nous montrons dans cet article que ce l'on ne peut pas faire cela:

\vspace{0.3 cm}

\noindent \textbf{Théorème \ref{Théorème principal}.} L'ensemble $(\mathcal{X}_3 \cap \mathfrak{h}_g)$ des réseaux symplectiques dans $\mathfrak{h}_g$ dont les systoles engendrent un sous-espace de dimension au moins 3 dans $\mathbb{R}^{2g}$ ne contient aucun modèle de type $\underline{E}\Sp(2g,\mathbb{Z})$. En particulier, il ne contient aucun rétract par déformation $\Sp(2g,\mathbb{Z})$-équivariant de $\mathfrak{h}_g$.

\vspace{0.3 cm}

Nous obtenons par la suite le même résultat en remplaçant $\Sp(2g,\mathbb{Z})$ par $\Aut(\SL(n,\mathbb{Z}))$:

\vspace{0.3 cm}

\noindent \textbf{Théorème \ref{Corollaire 1}.} L'ensemble $\mathcal{X}_3 \subset S_n=\SL(n,\mathbb{R})/\SO(n)$ ($n \geqslant 3$) des réseaux de $\mathbb{R}^{n}$ dont les systoles engendrent un sous-espace de dimension au moins 3 ne contient aucun modèle de type $\underline{E}\Aut(\SL(n,\mathbb{Z}))$.

\vspace{0.3 cm}

Ce corollaire est remarquable car $\SL(n,\mathbb{Z})$ et $\Aut(\SL(n,\mathbb{Z}))$ ne différent que d'un groupe fini et agissent tous les deux par isométries sur $S_n$.

\section{Introduction}

Let $\Gamma$ be an infinite discrete group. A \textit{model for $\underline{E}\Gamma$}, or a classifying space for proper actions, is a $\Gamma$-CW-complex $W$ such that for every subgroup $H \subset \Gamma$, the fixed point set $W^H$ is contractible if $H$ is finite and empty otherwise. Models for $\underline{E}\Gamma$ always exist, and the minimal possible dimension of such a model is the \textit{proper geometric dimension} of $\Gamma$, denoted $\underline{\gd}(\Gamma)$. Completing earlier work in \cite{Aramayona}, we proved in \cite{Lacoste} that if $G$ is a semisimple linear Lie group and $\Gamma \subset G$ is a lattice, then $\underline{\gd}(\Gamma)$ equals the \textit{virtual cohomological dimension} $\vcd(\Gamma)$ of $\Gamma$, that is the cohomological dimension of any torsionfree finite index subgroup of $\Gamma$. If $\Gamma$ is arithmetic and $K \subset G$ is a maximal compact subgroup, then $\vcd(\Gamma)$ equals the dimension of $G/K$ minus the rational rank of $\Gamma$ (see \cite{Borel-Serre}).

With the same notations, note that the symmetric space $X=G/K$ is itself a model for $\underline{E}\Gamma$, but not of minimal dimension unless $\Gamma$ is cocompact. It is then a question to find concretely a cocompact model $W$ for $\underline{E}\Gamma$ of dimension $\vcd(\Gamma)$. Besides the intrinsic interest of the problem, one can use such a model to compute the cohomology of $\Gamma$ (see examples in \cite{Soule}, \cite{Ash-Mcconnell}, \cite{Henn}, \cite{Dutour}). 

Because of the lack of another starting point, it is natural to try to construct $W$ as a subspace of the symmetric space $X$. In this case we call it a \textit{spine}.  More precisely a spine for $\Gamma$ is a $\Gamma$-equivariant deformation retract of the symmetric space $X=G/K$, of dimension $\vcd(\Gamma)$ and on which $\Gamma$ acts cocompactly. 

It might be surprising to the reader that such spines are known only in very few cases: basically only for $\mathbb{Q}$-rank 1 groups (see \cite{Yasaki}) and for $\SL(n,\mathbb{Z})$ (and somewhat more generally for linear symmetric spaces, see \cite{Soulethese}, \cite{Ash}, \cite{Souto-Pettet} and \cite{Pettet-Souto}) . The aim of this note is to maybe explain why it might not be easy to find spines.

First recall the construction of $\SL(n,\mathbb{Z})$'s spine. Indentify the symmetric space $S_n=\SL(n,\mathbb{R})/\SO(n)$ with the space of lattices in $\mathbb{R}^{n}$ of covolume 1 modulo isometries with a $\mathbb{Z}$-basis. The \textit{systole} of a lattice $\Lambda=A\mathbb{Z}^n$, with $A \in \SL(n,\mathbb{R})$, is defined as
\[ \syst(\Lambda)=\min_{v \in \mathbb{Z}^n \setminus \{0\}} |Av|.\]
We will also call systoles (or minimal vectors) the vectors $v$ which realize the minimum. A lattice is \textit{well-rounded} if its minimal vectors span $\mathbb{R}^n$. Generalizing a result of Soulé in \cite{Soulethese}, Ash proved in \cite{Ash} that the well-rounded retract, that is the set of all well-rounded lattices, is a spine for $\SL(n,\mathbb{Z})$. The idea to realize the retraction is as follows: given a non well-rounded lattice in $\mathbb{R}^n$, expand the space spanned by the shortest vectors and contract its orthogonal complement until we find an additional systole. In this way, one can prove that if $\mathcal{X}_i$ is the set of lattices whose systoles span a subspace of dimension at least $i$ (for $i=1,\dots, n$), then $\mathcal{X}_{i+1}$ is a $\SL(n,\mathbb{Z})$-equivariant deformation retract of $\mathcal{X}_i$ for every $i=1,\dots, n-1$. Remark that $\mathcal{X}_1=S_n$ and $\mathcal{X}_n$ is the set of well-rounded lattices, that is our well-rounded retract.

Some effort has been devoted to mimic the construction of the well-rounded retract in other situations. For instance in \cite{Ji} Ji constructed well-rounded retracts for mapping class groups acting on Teichmüller spaces, and  Bavard proved in \cite{Bavard2} that the symmetric space $\Sp(2g,\mathbb{R})/\U(g)$ (also known as the Siegel space $\mathfrak{h}_g$), identified with the set of symplectic lattices in $\mathbb{R}^{2g}$ endowed with a symplectic basis, admits a $\Sp(2g,\mathbb{Z})$-equivariant deformation retract, consisting of the set of symplectic lattices whose systoles span a non-isotropic subspace. In both cases, the retract has not minimal dimension. For example, in the case of $\Sp(2g,\mathbb{Z})$ the virtual cohomological dimension is $g^2$ and Bavard's retract has codimension one\footnote{Recall that in general, if $\Gamma \subset G$ is not cocompact, one can always construct a cocompact model for $\underline{E}\Gamma$ of codimension 1 in $G/K$, see Proposition 2.6 in \cite{Aramayona}.}: there are lattices there with only two systoles which are non-isotropic. To retract into a higher codimension set it would be reasonable to expect that one should be able to retract into the set of symplectic lattices with more linearly independant systoles. We prove that we cannot do this:
 
\begin{thm}\label{Théorème principal}
The set $(\mathcal{X}_3 \cap \mathfrak{h}_g)$ of symplectic lattices in $\mathfrak{h}_g$ whose systoles generate a vector space of dimension at least 3 in $\mathbb{R}^{2g}$ does not contain any model for $\underline{E}\Sp(2g,\mathbb{Z})$. In particular, it does not contain any $\Sp(2g,\mathbb{Z})$-equivariant deformation retract of $\mathfrak{h}_g$.
\end{thm}

We will also obtain that the same results holds if we replace $\Sp(2g,\mathbb{Z})$ by $\Aut(\SL(n,\mathbb{Z}))$:

\begin{thm}\label{Corollaire 1}
The set $\mathcal{X}_3 \subset S_n=\SL(n,\mathbb{R})/\SO(n)$ ($n \geqslant 3$) of lattices in $\mathbb{R}^{n}$ whose systoles generate a subspace of dimension greater than 3 does not contain any model for $\underline{E}\Aut(\SL(n,\mathbb{Z}))$.
\end{thm}

This second result is noteworthy because $\SL(n,\mathbb{Z})$ and $\Aut(\SL(n,\mathbb{Z}))$ only differ by a finite group and both act by isometries on $S_n$.

The remaining of this note is devoted to the proofs of Theorem \ref{Théorème principal} and Theorem \ref{Corollaire 1}.

\section{Proofs of the results}

We begin with the proof of Theorem \ref{Théorème principal}:

\begin{proof}[Proof of Theorem \ref{Théorème principal}]
 
Recall that the group $\Sp(2g,\mathbb{R})$ is the set of matrices $A$ of size $2g$ such that $^tAJA=J$ where $J$ is the block-diagonal matrix ($J_2,J_2,\dots, J_2)$ and $J_2=\begin{pmatrix}
0 &-1
\\ 1 & 0 
\end{pmatrix}$. In fact, $J$ is the matrix of a symplectic form $\omega$ of $\mathbb{R}^{2g}$ in the canonical basis. A basis of $\mathbb{R}^{2g}$ is said symplectic if the associated matrix is symplectic, which is equivalent to the fact that the matrix of the symplectic form $\omega$ in this basis is $J$. To prove the theorem, we will show that there exists a finite subgroup $H \subset \Sp(2g,\mathbb{Z})$ such that the fixed point set $W^H$ is not contractible, for every subset $W \subset (\mathcal{X}_3 \cap \mathfrak{h}_g)$.

 The result holds trivially for $g=1$, but let us recall some facts about the systole function in $\mathfrak{h}_1=\mathbb{H}^2$. First recall that we identify a point $\tau \in \mathbb{H}^2$ with the lattice generated by 1 and $\tau$ rescaled to have covolume 1. The well-rounded retract in $\mathbb{H}^2$ is then the Bass-Serre tree of $\SL(2,\mathbb{Z})$. Then, note that the lattices with maximal systole in $\mathbb{R}^2$ are the hexagonal ones. They are the translates by $\SL(2,\mathbb{Z})$ of the standard hexagonal lattice $\Lambda_0$, associated in $\mathbb{H}^2$ with $\tau_0=e^{i\frac{\pi}{3}}$. The point $\tau_0$ is the only fixed point by the homography $z \mapsto 1-\frac{1}{z}$ associated to the matrix $A_0=\begin{pmatrix}
1 & -1 \\ 1 & 0
\end{pmatrix} \in \SL(2,\mathbb{Z})$.

Let us continue with the case $g=2$. A way to construct a symplectic lattice in $\mathbb{R}^4$, with its canonical basis $(e_1, e_2, e_3, e_4)$, is to sum two lattices in $\vect{e_1,e_2}$ and $\vect{e_3,e_4}$, which are orthogonal for both the symplectic form $\omega$ and the usual euclidean scalar product. The corresponding element in $\mathfrak{h}_2=\Sp(4,\mathbb{R})/\U(2)$ belongs to a subspace homeomorphic to $\SL(2,\mathbb{R})/\SO(2)\times \SL(2,\mathbb{R})/\SO(2)$ which can be identified with the product of two copies of the hyperbolic plane $\mathbb{H}^2$. This subspace is also the fixed point set of the finite subgroup of $\Sp(4,\mathbb{Z})$ generated by the diagonal matrix  $(I_2,-I_2)$.

A lattice $\Lambda$ in $\mathbb{R}^4$ corresponding to an element in $\mathbb{H}^2 \times \mathbb{H}^2$ is a product of two lattices $\Lambda_1 \times \Lambda_2$ which are orthogonal, so the systoles of $\Lambda$ belong to the subset $(\Gamma_1 \times \{0\}) \cup (\{0\} \times \Gamma_2)$. $\Lambda$ has three linearly independant systoles if and only if $\Lambda_1$ and $\Lambda_2$ have the same systole and one of them is well-rounded. 

\vspace{0.5 cm}

\noindent \textbf{Claim}: The fixed point set $W^H$, with $H$ being the subgroup of $\Sp(4,\mathbb{Z})$ generated by $(I_2,-I_2)$ and $(I_2,A_0)$, is either empty or not connected.
\begin{proof}[Proof of the claim] The fixed point set of $\mathbb{H}^2 \times \mathbb{H}^2$ by the pair $(I_2,A_0) \in \SL(2,\mathbb{Z}) \times \SL(2,\mathbb{Z}) \subset \Sp(4,\mathbb{Z})$ is $\mathbb{H}^2 \times \{\tau_0\}$. If $\Lambda=\Lambda_1 \times \Lambda_0$ has 3 linearly independant systoles and lie in this set, then $\Lambda_0$ is hexagonal and has maximal systole, so $\Lambda_1$ has to be hexagonal too. It follows that $W^H$ is either empty or homeomorphic to the set of translates of $\tau_0$ by $\SL(2,\mathbb{Z})$ and hence is discrete.
\end{proof}
So $W^H$ is not contractible and $W$ is not a model for $\underline{E}\Sp(4,\mathbb{Z})$. We have proved the theorem in the case $g=2$.

For the general case, we will explain which finite subgroup $H$ of $\Sp(2g,\mathbb{Z})$ we will take. It will be generated by some finite order matrices in  $\SL(2,\mathbb{Z}) \times \dots \times \SL(2,\mathbb{Z}) \subset \Sp(2g,\mathbb{Z})$. First consider the $g$ diagonal matrices of the form $(I_2,I_2,\dots, -I_2, I_2,\dots, I_2)$. The fixed point set of the finite subgroup generated by them is $\mathbb{H}^2 \times \dots \times \mathbb{H}^2$. Add to this subgroup the matrix $(I_2, A_0,\dots, A_0)$. The fixed point set $(\mathfrak{h}_g)^H$ is then homeomorphic to $\mathbb{H}^2 \times \{\tau_0\} \times \dots \times \{\tau_0\}$ and we can apply the preceeding argument.
\end{proof}

\begin{rmq}
The symmetric space $\mathbb{H}^2 \times \mathbb{H}^2$ admits a $\SL(2,\mathbb{Z})\times \SL(2,\mathbb{Z})$-equivariant deformation retract of dimension $\vcd(\SL(2,\mathbb{Z})\times \SL(2,\mathbb{Z}))=2$ which is just the product of two copies of the well-rounded retract of $\mathbb{H}^2$, but the associated lattices in $\mathbb{R}^2 \times \mathbb{R}^2=\mathbb{R}^4$ are not well-rounded in general.
\end{rmq}

\begin{rmq}
It is worth mentioning that in the case $g=2$, MacPherson and McConnell have constructed in \cite{MacPherson} a \textit{weak spine} of $\mathfrak{h}_2$, that is, for every finite index torsionfree subgroup $\Gamma \subset \Sp(4,\mathbb{Z})$, a cocompact $\Gamma$-equivariant deformation retract $W_{\Gamma}$ of $\mathfrak{h}_2$ of dimension $\vcd(\Sp(4,\mathbb{Z}))=\cd(\Gamma)=4$. The methods they used are slightly different as the ones for the well-rounded retract, but do not yield a $\Sp(4,\mathbb{Z})$-equivariant deformation retract. They used the Voronoi decomposition of the symmetric space $\SL(4,\mathbb{R})/\SO(4)$ (identified with the set of positive definite quadratic forms in $\mathbb{R}^4$ modulo homotheties) and studied the intersection of the cells with $\mathfrak{h}_2=\Sp(4,\mathbb{R})/\U(2)$. It is the first example of a (weak) spine of a nonlinear symmetric space of real rank greater than 1. Note that we can also use the Voronoi sets of $\SL(2,\mathbb{R})/\SO(2)=\mathbb{H}^2$ to construct the well-rounded retract for $\SL(2,\mathbb{Z})$. See Chap.VII of \cite{Martinet} for more about the Voronoi sets. 
\end{rmq}

We continue with the proof of Theorem \ref{Corollaire 1}:

\begin{proof}[Proof of Theorem \ref{Corollaire 1}]
 Recall that for $n \geqslant 3$, the group $\Aut(\SL(n,\mathbb{Z}))$ is the semidirect product of the group of inner automorphism, which is isomorphic to $\PSL(n,\mathbb{Z})$, and the outer automorphism group generated by the conjugation by the diagonal matrix with entries $(-1,1,\dots,1)$ and the automorphism $\sigma$ defined by $\sigma(X)=(^tX)^{-1}$ for $X \in \SL(n,\mathbb{Z})$ (see \cite{Hua-Reiner}). The usual action of $\PSL(n,\mathbb{Z})$ on $S_n$ extends to an isometric action of $\Aut(\SL(n,\mathbb{Z}))$.
 
We begin with the case where $n=2p$ is even. We see that the Siegel space $\mathfrak{h}_p=\Sp(2p,\mathbb{R})/\U(p)$ is the fixed point set of the automorphism $\alpha$ defined by:
\[ \alpha(X)=\begin{pmatrix}
0 & -I_p \\ I_p & 0
\end{pmatrix} \sigma(X) \begin{pmatrix}
0 & I_p \\ -I_p & 0
\end{pmatrix}, \]
that is the composition of $\sigma$ and an inner automorphism. Then we can take the subgroup of $\Aut(\SL(2p,\mathbb{Z}))$ generated by $\alpha$ and the subgroup $H$ in the proof of Theorem \ref{Théorème principal} and apply the same argument as before.

If $n=2p+1$ is odd, we see that the fixed point set by the subgroup $\tilde{H}$ of $\Aut(\SL(n,\mathbb{Z}))$ generated by the inner automorphism of conjugation by the diagonal matrix $(1,-1,\dots,-1)$ and the outer automorphism $\tilde{\alpha}$ defined by:
\[ \tilde{\alpha}(X)=\begin{pmatrix}
1 & 0 & 0 \\
0 & 0 & -I_p \\ 0 & I_p & 0
\end{pmatrix} (^t X)^{-1} \begin{pmatrix}
1 & 0 & 0 \\
0 & 0 & I_p \\ 0 & -I_p & 0
\end{pmatrix}, \]
consists of all matrices of the form $\begin{pmatrix}
1 & 0 \\
0 & B
\end{pmatrix}$ with $B \in \Sp(2p,\mathbb{R})$. Then if we consider the finite subgroup $\hat{H}$ generated by $\tilde{H}$, the diagonal matrices $(1,I_2,I_2,\dots,-I_2,I_2,\dots,I_2)$ and $(1,A_0,\dots,A_0)$, the fixed point set is reduced to the lattice $\{1\} \times \Lambda_0 \times \dots \times \Lambda_0$ which has only one systole, so its intersection with $\mathcal{X}_3$ is empty.
\end{proof}

\begin{rmq}
The proof of Theorem \ref{Corollaire 1} only involves the outer automorphism $\sigma$. In fact, the well-rounded retract is also a spine for $\PGL(n,\mathbb{Z})$, which is of index 2 in $\Aut(\SL(n,\mathbb{Z}))$. Note also that in the case $n=2$, $\sigma$ is an inner automorphism and the Bass-Serre tree is the unique minimal spine for $\PGL(2,\mathbb{Z})$.
\end{rmq}

\begin{rmq}

It follows from the proof of Theorem \ref{Corollaire 1}, that if $W$ is a $\Aut(\SL(2g,\mathbb{Z}))$-equivariant deformation retract of $S_{2g}$, then its intersection with $\mathfrak{h}_g$ is a model for $\underline{E}\Sp(2g,\mathbb{Z})$. Then, to construct a spine for $\Sp(2g,\mathbb{Z})$, we could try to find one for $\Aut(\SL(2g,\mathbb{Z}))$. As we just saw, one cannot do this using just the systole function, but one can hope to succeed by using other classical functions on the space of lattices. For instance we can think to the $k$-systoles functions, which measure the volume of the $k$-dimensional sections of the lattice (see \cite{Bavard1} for definitions and properties). Note that like the usual systole, the $k$-systoles are exponentials of Busemann functions.

\end{rmq}

\vspace{0.5cm}

\noindent \textbf{Acknowledgements.} The author thanks Simon André, James Farre and Juan Souto for their relevant comments and corrections.

\bibliographystyle{plain}
\bibliography{biblio_siegel}

\hspace{0.5cm}

\textsc{IRMAR, Université de Rennes 1}

\textit{E-mail address}: \texttt{cyril.lacoste@univ-rennes1.fr}

\end{document}